\documentclass[12pt]{amsart}

\usepackage{amsmath}

\calclayout

\usepackage{amsfonts,amssymb,amsthm,enumitem}
\usepackage{mathtools}
\usepackage{nicefrac,setspace}
\usepackage{xcolor}

\usepackage{hyperref}       
\usepackage{url}            
\usepackage{booktabs}       
\usepackage{amsfonts}       
\usepackage{nicefrac}       
\usepackage{microtype}      

\newtheorem{theorem}{Theorem}

\newtheorem*{theorem*}{Theorem}
\newtheorem*{claim*}{Claim}
\newtheorem*{remark}{Remark}
\newtheorem*{lemma*}{Lemma}

\newtheorem{lemma}[theorem]{Lemma}

\newtheorem*{corollary*}{Corollary}

\newtheorem{prop}[theorem]{Proposition}

\newcommand{\R}{\mathbb{R}}

\newcommand{\E}{\mathop \mathbb{E}}

\title{A LYM inequality 
for product measures}
\author{Gal Yehuda}
\address{Department of Computer Science, Technion-IIT}
\email{ygal@technion.ac.il}
\author{Amir Yehudayoff}
\address{Department of Computer Science, The University of Copenhagen, and Department of Mathematics, Technion-IIT}
\email{yehudayoff@technion.ac.il}

\begin{document}

\begin{abstract}
This note proves a version of
Lubell--Yamamoto--Meshalkin inequality for general product measures.
\end{abstract}

\maketitle

\section{Introduction}
An antichain in $\{0,1\}^n$ 
is a family of sets with no pairwise strict inclusions.
Sperner's theorem provides a sharp upper bound on the size of antichains~\cite{sperner1928satz}.
The theorem states that if $A \subset \{0,1\}^n$ is an antichain
then 
$$|A| \leq \max_{\ell \in \{0,1,\ldots,n\}} \binom{n}{\ell} = {n \choose \lfloor n/2 \rfloor}.$$
It is fundamental and has many applications 
(specifically, in the context of anti-concentration; see~\cite{erdos1945lemma}).
There are several extensions of Sperner's theorem
to various settings (see e.g.~\cite{bollobas1986combinatorics,engel1997sperner,griggs1977three} and references within).
We prove a version of the
theorem for product measures.

Let $P$ be a product distribution on $\{0,1\}^n$.
We say that $P$ is non-trivial if $p_j(1-p_j) > 0$ for all $j \in [n]$,
where $p_j = \Pr_{z \sim P}[z_j = 1]$.
Aizenman, Germinet, Klein and Warzel
proved 
a version of Sperner's theorem for product distributions~\cite{aizenman2009bernoulli}.
They showed that the maximum measure of an antichain is
$\leq O(\tfrac{1}{\alpha \sqrt{n}})$ where
$\alpha = \min \{ p_1 ,\ldots,p_n , 1-p_1,\ldots,1-p_n\}$.
Their proof uses Bernoulli decompositions of random variables,
but it does not yield sharp bounds
(because it uses concentration bounds which are not entirely accurate).
We present a proof that gives obtain a sharp bound,
and develop some theory in the process (see Section~\ref{sec:antiVer}).

\begin{theorem}
\label{thm:sperner}
For every non-trivial product distribution $P$
on $\{0,1\}^n$, and for every antichain $A \subset \{0,1\}^n$,
\begin{align}
\label{eqn:Sper}\Pr_{z \sim P}[ z \in A] \leq \max_{\ell \in \{0,1,\ldots,n\}}
\Pr_{z \sim P}[|z|=\ell].
\end{align}
\end{theorem}

The Lubell–Yamamoto–Meshalkin inequality is stronger than
Sperner's theorem~\cite{bollobas1965generalized,lubell1966short,meshalkin1963generalization,yamamoto1954logarithmic}. It states that if $A \subset \{0,1\}^n$ is an antichain then
$$\sum_{\ell =0}^n \frac{|A_\ell|}{{n \choose \ell}} \leq 1,$$
where $A_\ell = \{ a \in A : |a|=\ell\}$.
The LYM inequality was also generalized in several ways (see e.g.~\cite{erdHos1992sharpening}),
and is sometimes more useful than Sperner's theorem (see e.g.~\cite{aizenman2009bernoulli}).
The following theorem is a strict generalization of the LYM inequality,
and puts the inequality in context.

\begin{theorem}
\label{thm:LYM}
For every non-trivial product distribution $P$
on $\{0,1\}^n$, and for every antichain $A \subset \{0,1\}^n$,
$$\sum_{\ell=0}^n 
\Pr_{z \sim P}[z \in A||z|=\ell]\leq 1.$$
\end{theorem}

What underlying property of product distributions allows to control
the measure of antichains?
We identify the following mechanism.
There is a way to sample a full chain
in a way that respects the measure.

\begin{lemma}
\label{lem:distChain}
Let $P$ be a non-trivial product distribution $P$
on $\{0,1\}^n$.
For $\ell \in \{0,1,\ldots,n\}$,
let $P_\ell$ be the distribution of $z \sim P$ conditioned on the event that
$|z| = \ell$.
Then, there is a distribution on maximal chains
$$\emptyset = c_0 \subset c_1 \subset \ldots \subset c_n = [n]$$
so that for every $\ell \in \{0,1,\ldots,n\}$,
the set $c_\ell$ is distributed according to $P_\ell$.
\end{lemma}

To make use of Theorem~\ref{thm:sperner},
we need to control the right hand side of~\eqref{eqn:Sper}.
Namely, we need to prove an anti-concentration result
for general product measures.
This is done in Section~\ref{sec:prodMeasure}.

\begin{theorem}
\label{thm:anticon}
There is a constant $C_1>0$ so that the following holds.
Let $P$ be a product distribution on $\{0,1\}^n$.
Denote the variance of $|z|$ for $z \sim P$ by
$\sigma^2_P = \sum_j p_j(1-p_j)$.
Then, 
$$\max_{\ell \in \{0,1,\ldots,n\}} \Pr_{z \sim P} [ |z|= \ell ] \leq \frac{C_1}{\sigma_P}.$$
\end{theorem}

\section{The measure of antichains}
\label{sec:antiVer}

Here we extend Sperner's theorem and the LYM inequality to
general product measures.
The main proposition is natural,
but the proof we found is technical.
Recall that $P_\ell$ denotes the distribution of
$z \sim P$ conditioned on $|z|=\ell$.

\begin{prop}
\label{prop:Probconst}
Let $P$ be a non-trivial product distribution on $\{0,1\}^n$.
Let $\ell \in \{0,1,\ldots,n-1\}$.
Then, there is a probability distribution on pairs $
(c_\ell,c_{\ell+1}) \in {[n] \choose \ell} \times 
{[n] \choose \ell+1}$ so that the following hold:
\begin{enumerate}
\item $\Pr[c_\ell \subset c_{\ell+1}]=1$. 
\item $c_\ell$ is distributed like $P_\ell$.
\item $c_{\ell+1}$ is distributed like $P_{\ell+1}$.
\end{enumerate}

\end{prop}

\begin{proof}
Choose $c_\ell$ according to $P_\ell$.
For $s \subset [n]$ of size $|s|=\ell$ and $j \not \in s$,
we need to decide what is the probability of $c_{\ell+1} = s \cup \{j\}$ conditioned on $c_\ell = s$.

Let $q \in \R^n$ be so that for all $j \in [n]$,
$$q_j = \frac{p_j}{1-p_j}>0.$$
Consider the symmetric polynomial
$$g_\ell(q) = \sum_{t \subseteq [n] : |t|=\ell} \prod_{j \in t} q_j ,$$
where $g_0(q)=1$.
For $s \subset [n]$ of size $\ell$,
\begin{align*}
\Pr[z=s] 
& = \prod_{a \in s} p_a \cdot \prod_{j \not \in s} (1-p_j) 
 = \prod_{j \in [n]} (1-p_j) \cdot \prod_{a \in s} q_a .
\end{align*}
It follows that
\begin{align*}
P_\ell(s) = \Pr_{z \sim P}[z=s||z|=\ell]
& = \frac{\prod_{a \in s} q_a}{g_\ell(q)}.
\end{align*}
Let  
$$h_{s,j}= \sum_{t \subset [n] : |t|=\ell, j \not \in t} \frac{1}{|(s \cup \{j\}) \setminus t|} \prod_{a \in t} q_a >0.$$
When $\ell=0$ set $h_{s,j} = 1$.
For every $s$,
\begin{align*}
\sum_{j \not \in s} q_j h_{s,j}
& = \sum_{j \not \in s} q_j\sum_{t: |t|=\ell, j \not \in t} \frac{1}{|(s \cup \{j\}) \setminus t|} \prod_{a \in t} q_a \\
& = \sum_{t: |t|=\ell} \sum_{j \not \in s \cup t}  \frac{1}{|(s \cup \{j\}) \setminus t|} \prod_{a \in t \cup \{j\}} q_a \\
& = \sum_{t: |t|=\ell} \sum_{j \not \in s \cup t}  \frac{1}{1+|s \setminus t|} \prod_{a \in t \cup \{j\}} q_a \\
& = \sum_{t': |t'|=\ell+1} \prod_{a \in t'} q_a \sum_{j \in t' , j \not \in s}  \frac{1}{1+|s \setminus t'|}  \tag{$t' = t \cup \{j\}$} .
\end{align*}
Because $|s|=\ell$ and $|t'|=\ell+1$,
$$|t' \setminus s|+|t' \cap s|=\ell+1 =
 |s \setminus t'| + |s \cap t'|+1.$$
We can conclude
\begin{align*}
\sum_{j \not \in s} q_j h_{s,j}
& = \sum_{t': |t'|=\ell+1} \prod_{a \in t'} q_a \sum_{j \in t' , j \not \in s}  \frac{1}{|t' \setminus s|} = g_{\ell+1}(q) .
\end{align*}
This is also true for $\ell=0$.

Finally, define 
$$\Pr[c_{\ell+1} = s \cup \{j\} | c_\ell = s] = \frac{q_j h_{s, j}}{g_{\ell+1}(q)}.$$
We need to show that $c_{\ell+1}$ is properly distributed.
For fixed $s'$ of size $\ell+1$,
\begin{align*}
\sum_{j \in s'} h_{s' \setminus \{j\},j}
& = \sum_{j \in s'} \sum_{t: |t|=\ell, j \not \in t} \frac{1}{|s' \setminus t|} \prod_{a \in t} q_a \\
& =  \sum_{t: |t|=\ell} \prod_{a \in t} q_a \sum_{j \in s' \setminus t} \frac{1}{|s' \setminus t|}   = g_\ell(q) .
\end{align*}
This is also true for $\ell=0$.
So,
\begin{align*}
\Pr[c_{\ell+1}=s']
& = \sum_{j \in  s'} \Pr[c_\ell = s' \setminus \{j\}] 
\frac{q_j \cdot h_{s' \setminus \{j\},j} }{g_{\ell+1}(q)} \\
& = \frac{\prod_{a \in s'} q_a}{g_\ell(q){g_{\ell+1}(q)} } \sum_{j \in  s'}
h_{s' \setminus \{j\},j} = P_{\ell+1}( s') .
\end{align*}
\end{proof}

\begin{proof}[Proof of Lemma~\ref{lem:distChain}]
The set $c_0$ is fixed to be empty.
For $\ell < n$, define $c_{\ell+1}$ from $c_{\ell}$
via Proposition~\ref{prop:Probconst}. 
\end{proof}

\begin{proof}[Proof of Theorem~\ref{thm:LYM}]
Let $C = \{c_0 , c_1,\ldots, c_n\}$ be a random maximal chain
as in Lemma~\ref{lem:distChain}.
Let $L$ be the number of $\ell \in \{0,1,\ldots,n\}$ so that $c_\ell \in A$.
Because $A$ is an antichain, $\Pr[L \leq 1]=1$.
On the other hand, 
\begin{align}
\label{eqn:LYMp}
\E[L] 
 & = \sum_\ell \sum_{a \in A:|a|=\ell} \Pr[c_\ell=a] \\
\notag & = \sum_\ell \sum_{a \in A:|a|=\ell} \Pr\big[z=a\big| |z|=\ell \big] \\
\notag & = \sum_\ell \Pr[z \in A||z|=\ell].
\end{align}

\end{proof}

\begin{proof}[Proof of Theorem~\ref{thm:sperner}]
By Theorem~\ref{thm:LYM},
\begin{align*}
\Pr[z \in A] 
& = \sum_{\ell} \Pr[|z|=\ell] \Pr[z \in A|  |z|=\ell] \leq \max_\ell \Pr[|z|=\ell]. 
\end{align*}
\end{proof}

\section{Anti-concentration}
\label{sec:prodMeasure}

Here we prove a general anti-concentration
result for product measures.
The simple proof is inspired by~\cite{rao2018anti}.

\begin{proof}[Proof of Theorem~\ref{thm:anticon}]
Think of $z$ as taking values in $\{\pm 1\}^n$;
this just simplifies the calculations. 
Let $\theta$ be uniformly distributed in $[0,1]$.
For every integer $t$,
\begin{align*}
\Pr\Big[\sum_{j=1}^n z_j =t\Big]
& = \E_\theta \E_z \exp \Big( 2 \pi i \theta \Big(\sum_j z_j-t\Big)\Big) \\
& \leq \E_\theta \Big| \prod_j \E_{z_j} \exp(2 \pi i \theta z_j) \Big| .
\end{align*}
For each $j$,
$\E_{z_j} \exp({2 \pi i \theta z_j})
= p_j \exp({2\pi i \theta})+(1-p_j) \exp(-2 \pi i \theta)$. 
So,
\begin{align*}
|\E_{z_j} \exp(2 \pi i \theta z_j)|^2
& = p_j^2 + (1-p_j)^2 + 2 p_j (1-p_j) \cos(2 \pi \theta) \\ 
& = 1 - 2 p_j (1-p_j) (1-\cos(2 \pi \theta)) \\
& = 1 - 4 p_j (1-p_j) \sin^2( \pi \theta) \\
& \leq \exp(-4 p_j (1-p_j) \sin^2( \pi \theta) ).
\end{align*}
It follows that 
\begin{align*}
\Pr\Big[\sum_{j=1}^n z_j =t\Big]
& \leq \E_\theta \exp (-2 \sigma^2_P \sin^2( \pi \theta)  ) \\
& = 2\int_0^{1/2}   \exp (-2 \sigma^2_P \sin^2( \pi \theta)) d \theta \\
& \leq \int_{-\infty}^{\infty}   \exp (-  \sigma^2_P  \theta^2 ) d \theta 
\tag{$\sin(\xi) \geq \xi/2$} \leq \frac{C_1}{\sigma_P}.
\end{align*}

\end{proof}

\begin{remark}
A more careful asymptotic analysis yields the sharp bound
$C_1 \leq \frac{1}{\sqrt{2 \pi}} + o(1)$.
\end{remark}

\bibliographystyle{abbrv}
\bibliography{main}

\end{document}